\documentclass[12pt, reqno]{amsart}
\usepackage{amsmath, amsthm, amscd, amsfonts, amssymb, graphicx, color}
\usepackage[bookmarksnumbered, colorlinks, plainpages]{hyperref}
\hypersetup{colorlinks=true,linkcolor=red, anchorcolor=green, citecolor=cyan, urlcolor=red, filecolor=magenta, pdftoolbar=true}

\newtheorem{theorem}{Theorem}[section]
\newtheorem{lemma}[theorem]{Lemma}

\newtheorem{cor}[theorem]{Corollary}
\theoremstyle{definition}

\theoremstyle{remark}
\newtheorem{remark}[theorem]{\bf{Remark}}
\numberwithin{equation}{section}
\begin{document}

\title [Sharp inequalities for the numerical radius]{Sharp inequalities for the numerical radius of Hilbert space operators and operator matrices} 

\author{Pintu Bhunia, Kallol Paul and Raj Kumar Nayak}

\address{(Bhunia) Department of Mathematics, Jadavpur University, Kolkata 700032, West Bengal, India}
\email{pintubhunia5206@gmail.com}

\address{(Paul) Department of Mathematics, Jadavpur University, Kolkata 700032, West Bengal, India}
\email{kalloldada@gmail.com}

\address{(Nayak) Department of Mathematics, Jadavpur University, Kolkata 700032, West Bengal, India}
\email{rajkumarju51@gmail.com}

%\thanks will become a 1st page footnote.
\thanks{First and third author would like to thank UGC, Govt. of India for the financial support in the form of JRF. Prof. Kallol Paul would like to thank RUSA 2.0, Jadavpur University for the partial support.}
%\thanks{}
%\thanks{}
%    Information for second author

%    General info


\date{}
\maketitle
\begin{abstract}
We present new upper and lower bounds for the numerical radius of a bounded linear operator defined on a  complex  Hilbert  space,  which improve on the existing bounds. Among many other  inequalities proved in this article, we show that for a non-zero bounded linear operator $T$ on a Hilbert space $H,$   $w(T)\geq \frac{\|T\|}{2}+\frac{m(T^2)}{2\|T\|},  $ where $w(T)$ is the numerical radius of $T$ and $m(T^2)$ is the Crawford number of $T^2$.  This substantially improves on the existing inequality $w(T)\geq \frac{\|T\|}{2} .$ We also  obtain some upper and lower bounds for the numerical radius of operator matrices and illustrate with numerical examples that these bounds are better  than the existing bounds. 
\end{abstract}

\section{\textbf{Introduction}}

\noindent Computation of the numerical radius of a bounded linear operator defined on complex Hilbert spaces is an interesting embroiled problem. Till date one  can compute the exact numerical radius for certain special class of operators and  for this reason estimation of  bounds  of the numerical radius  is a very important  problem. Our aim in this article to present better estimation of  the numerical radius of bounded linear operators and operator matrices. The following notations and terminologies are necessary to begin with.
Let $\mathbb{H}_1$ and $\mathbb{H}_2$ be two complex Hilbert spaces with inner product $\langle.,.\rangle.$  Let $B(\mathbb{H}_1,\mathbb{H}_2)$ denote the set of all bounded linear operators from $\mathbb{H}_1$ into $\mathbb{H}_2,$  if $\mathbb{H}_1=\mathbb{H}_2=\mathbb{H}$(say), then we write $B(\mathbb{H}_1,\mathbb{H}_2)=B(\mathbb{H}).$  For $T\in B(\mathbb{H})$, let $\|T\|$ and $c(T)$ denotes the usual operator norm and minimum norm of $T$ respectively, defined as 
\begin{eqnarray*}
\|T\|&=&\sup \left\{\|Tx\|:x\in \mathbb{H}, \|x\|=1 \right \} ~~\mbox{and}\\
 c(T)&=&\inf \left\{\|Tx\|:x\in \mathbb{H}, \|x\|=1 \right \},  
\end{eqnarray*}
where $\|.\|$ is the norm on $\mathbb{H}$ induced from the inner product $\langle.,.\rangle.$ Let $\sigma(T)$ denote the spectrum  of $T$, and $ r(T)$, the spectral radius of $T$, defined as 
\[r(T)=\sup \left\{ |\lambda|: \lambda \in \sigma(T)\right \}.\]
The numerical range  of $T,$  denoted as $W(T)$, is defined as 
\[W(T)=\left\{\langle Tx,x \rangle :x\in \mathbb{H}, \|x\|=1 \right \}.\]
 The numerical radius $w(T)$ and Crawford number $m(T)$ are defined as
\begin{eqnarray*} 
w(T)&=&\sup \left\{ |\lambda|: \lambda \in W(T)\right \} ~~\mbox{and}\\
m(T)&=&\inf \left\{ |\lambda|: \lambda \in W(T)\right \}.
\end{eqnarray*}
It is well known that the numerical range is a convex subset of the scalar field and closure of the numerical range contains the spectrum, i.e.,  $\sigma(T)\subseteq \overline{W(T)}$, so $r(T)\leq w(T).$  The numerical radius  $w(.)$ acts as  a norm on $B(\mathbb{H})$ and is equivalent to the operator norm $\|.\|$ satisfying the following inequality
\[\frac{\|T\|}{2} \leq \max \left \{r(T), \frac{\|T\|}{2} \right\}\leq w(T)\leq \|T\|.\]
For further  properties of numerical range and numerical radius, we refer  reader to \cite{B,GR}. \\

\noindent Over the years many  eminent mathematicians have studied and improved on the above inequality, to cite a few of them are  \cite{D,GK,HKS,KMY,K,S,Y}.  Recently we \cite{BBP1,BBP,BBP2,PB2,PB} have developed some bounds for the numerical radius and applied them to estimate zeros of polynomials. In 1963, Bernau and Smithies \cite{BS} gave an elegant proof of the inequality $ w(T) \geq \frac{1}{2} \|T\|$ using parallelogram law. In this paper we improve on this inequality  to prove that $ w(T) \geq \frac{1}{2} \|T\| + \frac{m(T^2)}{2\|T\|}$. We  generalize the  inequality \cite[Lemma 3]{BS} substantially to obtain new inequalities for the  numerical radius.  We further obtain bounds for the numerical radius of an $n \times n$ operator matrix $T$, where $ T =(A_{ij})$  is defined on the complex Hilbert space $\mathbb{H}=\mathbb{H}_1 \bigoplus \mathbb{H}_2 \bigoplus \ldots \bigoplus \mathbb{H}_n, $ 
where $\mathbb{H}_i$ $(i=1,2, \ldots, n) $ are complex Hilbert spaces. We show that the bounds obtained here  improve on and generalize  the existing bounds given in \cite{HKS,KMY}.

\section{\textbf{Inequalities for the numerical radius of product of operators}}

\noindent  We begin this section with the following inequality proved in  \cite[Lemma 3]{BS}. 

\begin{lemma} \label{lemma:B and S}
	Let $T \in B(\mathbb{H}).$ Then  for all $x\in \mathbb{H}$ 
	\begin{eqnarray}\label{number 1}
	\|Tx\|^2+|\langle T^2x,x\rangle| &\leq& 2w(T)\|Tx\| \|x\|.
	\end{eqnarray}
\end{lemma}

 We genralize this inequality in the following lemma.

\begin{lemma} \label{lemma:G B and S}
Let $A, T, B \in B(\mathbb{H}).$ Then,  for all $x\in \mathbb{H}$
\begin{eqnarray}\label{number G1}
|\langle A^*TBx,x\rangle|+|\langle B^*TAx,x\rangle|\leq 2w(T)\|Ax\|\|Bx\|.
\end{eqnarray}
\end{lemma}

\begin{proof} 
Let $x\in \mathbb{H}$ and ${\theta, \phi}$ be real numbers such that $e^{i\phi}\langle B^*TAx,x \rangle=|\langle B^*TAx,x \rangle|$, $e^{2i\theta}\langle e^{-i\phi}A^*TBx,x \rangle=|\langle e^{-i\phi}A^*TBx,x \rangle|=|\langle A^*TBx,x \rangle|.$
Then for non-zero real number $\lambda$, we have

\begin{eqnarray*}
&& 2e^{2i\theta}\langle TBx,e^{i\phi}Ax \rangle+2 e^{i\phi}\langle TAx,Bx \rangle\\ 
&& \hspace{.8 cm}=\langle e^{i\theta}T\left(\lambda e^{i\theta}Bx+\frac{1}{\lambda}e^{i\phi}Ax \right), \lambda e^{i\theta}Bx+\frac{1}{\lambda}e^{i\phi}Ax \rangle\\ 
&& \hspace{1.8 cm} - \langle e^{i\theta}T\left(\lambda e^{i\theta}Bx-\frac{1}{\lambda}e^{i\phi}Ax\right), \lambda e^{i\theta}Bx-\frac{1}{\lambda}e^{i\phi}Ax \rangle\\
\Rightarrow && 2e^{2i\theta}\langle e^{-i\phi}A^*TBx,x \rangle+2 e^{i\phi}\langle B^*TAx,x \rangle\\ 
&& \hspace{.8 cm}=\langle e^{i\theta}T\left(\lambda e^{i\theta}Bx+\frac{1}{\lambda}e^{i\phi}Ax \right), \lambda e^{i\theta}Bx+\frac{1}{\lambda}e^{i\phi}Ax \rangle\\ 
&& \hspace{1.8 cm} - \langle e^{i\theta}T\left(\lambda e^{i\theta}Bx-\frac{1}{\lambda}e^{i\phi}Ax\right), \lambda e^{i\theta}Bx-\frac{1}{\lambda}e^{i\phi}Ax \rangle\\
\Rightarrow && 2\left|\langle A^*TBx,x \rangle \right|+2 \left |\langle B^*TAx,x \rangle \right |\\ 
&& \hspace{.8 cm}=\langle e^{i\theta}T\left(\lambda e^{i\theta}Bx+\frac{1}{\lambda}e^{i\phi}Ax \right), \lambda e^{i\theta}Bx+\frac{1}{\lambda}e^{i\phi}Ax \rangle\\
&& \hspace{1.8 cm} - \langle e^{i\theta}T\left(\lambda e^{i\theta}Bx-\frac{1}{\lambda}e^{i\phi}Ax\right), \lambda e^{i\theta}Bx-\frac{1}{\lambda}e^{i\phi}Ax \rangle
\end{eqnarray*}
\begin{eqnarray*}
\Rightarrow && 2\left|\langle A^*TBx,x \rangle \right|+2 \left |\langle B^*TAx,x \rangle \right |\\ 
&& \hspace{.8 cm}\leq \left |\langle e^{i\theta}T\left(\lambda e^{i\theta}Bx+\frac{1}{\lambda}e^{i\phi}Ax \right), \lambda e^{i\theta}Bx+\frac{1}{\lambda}e^{i\phi}Ax \rangle\right |\\
&& \hspace{1.8 cm} + \left |\langle e^{i\theta}T\left(\lambda e^{i\theta}Bx-\frac{1}{\lambda}e^{i\phi}Ax\right), \lambda e^{i\theta}Bx-\frac{1}{\lambda}e^{i\phi}Ax \rangle \right|\\
\Rightarrow && 2\left|\langle A^*TBx,x \rangle \right|+2 \left |\langle B^*TAx,x \rangle \right |\\  
&&  \hspace{.8 cm}\leq w(T) \left( \left\|\lambda e^{i\theta}Bx+\frac{1}{\lambda}e^{i\phi}Ax\right\|^2+ \left\|\lambda e^{i\theta}Bx-\frac{1}{\lambda}e^{i\phi}Ax\right\|^2 \right)\\
\Rightarrow && \left|\langle A^*TBx,x \rangle \right|+ \left |\langle B^*TAx,x \rangle \right | \leq w(T) \left( \lambda^2 \|Bx\|^2+\frac{1}{\lambda^2}\|Ax\|^2 \right).
\end{eqnarray*}
This holds for all non-zero real $\lambda.$ If $\|Bx\| \neq 0,$ then we choose $\lambda^2=\frac{\|Ax\|}{\|Bx\|}.$  So, we get
\begin{eqnarray*}
|\langle A^*TBx,x\rangle|+|\langle B^*TAx,x\rangle|\leq 2w(T)\|Ax\|\|Bx\|.
\end{eqnarray*}
Clearly this inequality holds also when $\|Bx\|=0.$ This completes the proof of the lemma.
\end{proof}
\begin{remark}
If  we take $A=T$ and $B=I$ in Lemma \ref{lemma:G B and S}, then we get the inequality \cite[Lemma 3]{BS}. 
	
\end{remark}

Now using the inequality in Lemma \ref{lemma:G B and S}, we obtain the following inequalities involving numerical radius, Crawford number and operator norm of bounded linear operators.  

\begin{theorem}\label{theorem: G1}
Let $A, T, B \in B(\mathbb{H}).$ Then the following inequalities holds:
\begin{eqnarray*}
m(A^*TB)+w(B^*TA) &\leq & 2w(T)\|A\|\|B\|,
\end{eqnarray*}
\begin{eqnarray*}
w(A^*TB)+m(B^*TA) &\leq & 2w(T)\|A\|\|B\|.
\end{eqnarray*}
\end{theorem}

\begin{proof}
Taking $\|x\|=1$ in the inequality (\ref{number G1}), we have
\begin{eqnarray*}
|\langle A^*TBx,x\rangle|+|\langle B^*TAx,x\rangle|&\leq& 2w(T)\|A\|\|B\|\\
\Rightarrow m(A^*TB)+|\langle B^*TAx,x\rangle| &\leq& 2w(T)\|A\| \|B\|.
\end{eqnarray*}
Taking supremum over $\|x\|=1$, we get
\[m(A^*TB)+w(B^*TA) \leq 2w(T)\|A\| \|B\|. \] 
Again taking $\|x\|=1$ in the inequality (\ref{number G1}),  we have
\begin{eqnarray*}
|\langle A^*TBx,x\rangle|+|\langle B^*TAx,x\rangle|&\leq& 2w(T)\|A\|\|B\|\\
\Rightarrow |\langle A^*TBx,x\rangle| + m(B^*TA)+ &\leq& 2w(T)\|A\| \|B\|.
\end{eqnarray*}
Taking supremum over $\|x\|=1$, we get
\[w(A^*TB)+m(B^*TA) \leq 2w(T)\|A\| \|B\|. \]
This completes the proof of the theorem.
\end{proof}

Taking $B=I, T=A$ and $A=B$ in the above Theorem \ref{theorem: G1}, we get the following upper bounds for the numerical radius of product of two operators, which improve on the existing bounds.

\begin{cor}\label{cor: G12}
Let $A,B \in B(\mathbb{H}).$ Then the following inequalities holds:
\begin{eqnarray*}
w(AB) &\leq& 2w(A)\|B\|-m(B^*A), \\\\
 w(AB) &\leq& 2w(B)\|A\|-m(BA^*).
\end{eqnarray*}
\end{cor}
 
\begin{remark}
It is clear that both the inequalities obtained in Corollary \ref{cor: G12} improves on the inequalities $w(AB)\leq 2w(A)\|B\|\leq 4w(A)w(B)$ and $w(AB)\leq 2w(B)\|A\|\leq 4w(A)w(B)$ respectively, (see \cite[Th. $2.5$-$2$]{GR}). 
\end{remark}

Next using the above Lemma \ref{lemma:B and S}, we establish some new inequalities for the numerical radius of $2 \times 2$ operator matrices with zero operators as main diagonal entries. 
 
\begin{theorem}\label{theorem: G2}
Let $A, B \in B(\mathbb{H}).$ Then the following inequalities holds:
\begin{eqnarray*}
(i)~~\|A\|^2+ m(BA) &\leq& 2w\left(\begin{array}{cc}
    0&A \\
    B&0
 \end{array}\right) \|A\|,
\end{eqnarray*}
\begin{eqnarray*}
(ii)~~c^2(A)+ w(BA) &\leq& 2w\left(\begin{array}{cc}
    0&A \\
    B&0
 \end{array}\right) \|A\|,
\end{eqnarray*}
\begin{eqnarray*}
(iii)~~\|B\|^2+ m(AB) &\leq& 2w\left(\begin{array}{cc}
    0&A \\
    B&0
 \end{array}\right) \|B\|,
\end{eqnarray*}
\begin{eqnarray*}
(iv)~~c^2(B)+ w(AB) &\leq& 2w\left(\begin{array}{cc}
    0&A \\
    B&0
 \end{array}\right) \|B\|.
\end{eqnarray*}
\end{theorem}

\begin{proof}
Putting  $T=\left(\begin{array}{cc}
    0&A \\
    B&0
\end{array}\right)\in B(\mathbb{H} \oplus \mathbb{H})$ and $x=(x_1,x_2)^t \in \mathbb{H} \oplus \mathbb{H}$ with $\|x\|=1$, i.e., $\|x_1\|^2+\|x_2\|^2=1$ in the  inequality (\ref{number 1}), we get
\begin{eqnarray*}
\|Ax_2\|^2+\|Bx_1\|^2+|\langle ABx_1,x_1\rangle+\langle BAx_2,x_2\rangle| &\leq &   2w (T) \left( \|Ax_2\|^2+\|Bx_1\|^2\right)^{\frac{1}{2}}.
\end{eqnarray*}
Taking $x_1=0$, we get
\begin{eqnarray*}
\|Ax_2\|^2+|\langle BAx_2,x_2\rangle| &\leq &   2 w \left(\begin{array}{cc}
    0&A \\
    B&0
\end{array}\right)  \|Ax_2\|\\
\Rightarrow \|Ax_2\|^2+|\langle BAx_2,x_2\rangle| &\leq &   2 w \left(\begin{array}{cc}
    0&A \\
    B&0
\end{array}\right)  \|A\|\\
\Rightarrow \|Ax_2\|^2+m(BA) &\leq &   2 w \left(\begin{array}{cc}
    0&A \\
    B&0
\end{array}\right)  \|A\|
\end{eqnarray*}
Taking supremum over $\|x_2\|=1$, we get the inequality (i), i.e., 
\begin{eqnarray*}
\|A\|^2+m(BA) &\leq &   2 w \left(\begin{array}{cc}
    0&A \\
    B&0
\end{array}\right)  \|A\|.
\end{eqnarray*}
Again from the inequality 
\begin{eqnarray*}
\|Ax_2\|^2+|\langle BAx_2,x_2\rangle| &\leq &   2 w \left(\begin{array}{cc}
    0&A \\
    B&0
\end{array}\right)  \|A\|, ~~\mbox{we get}
\end{eqnarray*} 
\begin{eqnarray*}
c^2(A)+|\langle BAx_2,x_2\rangle| &\leq &   2 w \left(\begin{array}{cc}
    0&A \\
    B&0
\end{array}\right)  \|A\|.
\end{eqnarray*}
Taking supremum over $\|x_2\|=1$, we get the inequality (ii), i.e.,
\begin{eqnarray*}
c^2(A) + w(BA) &\leq &   2 w \left(\begin{array}{cc}
    0&A \\
    B&0
\end{array}\right) \|A\|.
\end{eqnarray*}
Similarly taking $x_2=0$ and supremum over $ \|x_1\|=1,$ we can prove the remaining inequalities.
\end{proof}

Next taking $A=B=T$ in  Theorem \ref{theorem: G2} and using  the equality $w \left(\begin{array}{cc}
    0&A \\
    A&0
\end{array}\right)=w(A)$, we get the following lower bounds for the numerical radius of non-zero bounded linear operators.

\begin{theorem}\label{theorem:lower bounds}
Let $T\in B(\mathbb{H})$ be non-zero. Then the following inequalities holds:
\begin{eqnarray}\label{number 2}
w(T)&\geq& \frac{\|T\|}{2}+\frac{m(T^2)}{2\|T\|}, 
\end{eqnarray}
\begin{eqnarray}\label{number 3}  
w(T)&\geq& \frac{c^2(T)}{2\|T\|}+\frac{w(T^2)}{2\|T\|}.
\end{eqnarray}
\end{theorem}

\begin{remark}
The inequality (\ref{number 2}) improves on the existing inequality $w(T)\geq \frac{\|T\|}{2} $ substantially. Also from the inequality (\ref{number 2}), it follows   that if $w(T)= \frac{\|T\|}{2}$ then $m(T^2)=0$. There are operators for which $m(T^2)=0$ but $w(T) \neq  \frac{\|T\|}{2}.$  
\end{remark}

Next we prove a necessary and sufficient condition for $w(T)=\frac{\|T\|}{2},$ where $T$ is an $n\times n$ complex matrix.

\begin{theorem}
Let $T$ be an $n\times n$ complex matrix. Then $w(T)=\frac{\|T\|}{2}$ if and only if $T$ is unitarily similar to a matrix of the form $\left(\begin{array}{cc}
    0&\|T\| \\
    0&0
\end{array}\right) \oplus \|T\| B$ where $B$ is a matrix of order $n-2$ and $w(B)\leq \frac{1}{2}.$
\end{theorem}

\begin{proof}
Necessary part of this theorem follows from \cite[Th. $1.3$-$5$]{GR} and sufficient part of this theorem is obvious.
\end{proof}

\begin{remark}\label{remark 1} 
The inequalities (\ref{number 2}) and (\ref{number 3})  obtained by us in Theorem \ref{theorem:lower bounds} are incomparable. Consider  $T=\left(\begin{array}{cc}
    0&1 \\
    0&0
 \end{array}\right),$ then it is easy to see that , (\ref{number 2}) gives $w(T)\geq \frac{1}{2}$ and (\ref{number 3}) gives  $w(T)\geq 0,$ whereas if we consider $T=\left(\begin{array}{cc}
    i&0 \\
    0&1
 \end{array}\right),$ then  (\ref{number 2}) gives $w(T)\geq \frac{1}{2}$ and (\ref{number 3}) gives  $w(T)\geq 1.$  
\end{remark}

Using Theorem \ref{theorem:lower bounds} and  noting the Remark \ref{remark 1}, we obtain the following lower bound for the numerical radius of non-zero bounded linear operators.
\begin{cor}
Let $T\in B(\mathbb{H})$ be non-zero. Then  
\begin{eqnarray*}
w(T)&\geq& \frac{1}{2\|T\|}\max \big\{ \|T\|^2+m(T^2), c^2(T)+w(T^2) \big\}.
\end{eqnarray*}
\end{cor}

Next we prove another inequality for the numerical radius in terms of sum of product of operators.

\begin{lemma} \label{lemma: G3}
Let $A, T, B \in B(\mathbb{H}).$ Then, for all $x\in \mathbb{H}$	
\begin{eqnarray*}
\left |\langle \left (A^*TB \pm B^*TA \right)x,x \rangle \right| &\leq& 2w(T)\|Ax\| \|Bx\|.
\end{eqnarray*}
\end{lemma}

\begin{proof}
Let $\theta$ and $\lambda (\neq 0)$ be real numbers. Then
\begin{eqnarray*}
2\langle e^{2i\theta}TBx,Ax \rangle +2 \langle TAx,Bx \rangle &=& \langle e^{i\theta}T \left( \lambda e^{i\theta}Bx+\frac{1}{\lambda}Ax\right),\lambda e^{i\theta}Bx+\frac{1}{\lambda}Ax \rangle \\
 && - \langle e^{i\theta}T \left( \lambda e^{i\theta}Bx-\frac{1}{\lambda}Ax\right),\lambda e^{i\theta}Bx-\frac{1}{\lambda}Ax \rangle\\
 \Rightarrow 2\langle e^{2i\theta}A^*TBx,x \rangle +2 \langle B^*TAx,x \rangle &=& \langle e^{i\theta}T \left( \lambda e^{i\theta}Bx+\frac{1}{\lambda}Ax\right),\lambda e^{i\theta}Bx+\frac{1}{\lambda}Ax \rangle \\
 && - \langle e^{i\theta}T \left( \lambda e^{i\theta}Bx-\frac{1}{\lambda}Ax\right),\lambda e^{i\theta}Bx-\frac{1}{\lambda}Ax \rangle\\
 \Rightarrow 2\langle \left (e^{2i\theta}A^*TB + B^*TA \right)x,x \rangle &=& \langle e^{i\theta}T \left( \lambda e^{i\theta}Bx+\frac{1}{\lambda}Ax\right),\lambda e^{i\theta}Bx+\frac{1}{\lambda}Ax \rangle \\
 && - \langle e^{i\theta}T \left( \lambda e^{i\theta}Bx-\frac{1}{\lambda}Ax\right),\lambda e^{i\theta}Bx-\frac{1}{\lambda}Ax \rangle\\
 \Rightarrow \left | \langle \left (e^{2i\theta}A^*TB + B^*TA \right)x,x \rangle \right | &\leq & w(T) \left ( \| \lambda e^{i\theta}Bx+\frac{1}{\lambda}Ax\|^2+ \| \lambda e^{i\theta}Bx-\frac{1}{\lambda}Ax \|^2  \right)\\
 &=& w(T) \left( \lambda^2 \|Bx\|^2+\frac{1}{\lambda^2}\|Ax\|^2 \right)\\
 \end{eqnarray*}
Since $\lambda$ is arbitrary non-zero real number, so we choose $\lambda^2=\frac{\|Ax\|}{\|Bx\|}, Bx \neq 0$. Therefore, we get
\[\left | \langle \left (e^{2i\theta}A^*TB + B^*TA \right)x,x \rangle \right | \leq 2 w(T)\|Ax\| \|Bx\|. \]
Clearly this holds also when $Bx=0.$  Since $\theta$ is arbitrary real number, so we take $\theta=0$ and $\theta=\frac{\pi}{2}$ respectively, and we get the desired inequality of the Lemma.
 \end{proof}

Using  Lemma \ref{lemma: G3}, we obtain the following inequality.  

\begin{theorem} \label{theorem:HKS}
Let $A, T, B \in B(\mathbb{H}).$ Then
\[w(A^*TB \pm B^*TA) \leq  2 w(T)\|A\|\|B\|.\]
\end{theorem}
\begin{proof}
Taking  $\|x\|=1$ in the inequality of Lemma \ref{lemma: G3}, we get
\begin{eqnarray*}
\left |\langle \left (A^*TB \pm B^*TA \right)x,x \rangle \right| &\leq& 2w(T)\|A\| \|B\|.
\end{eqnarray*}
Now taking supremum over $\|x\|=1$,	we get the required inequality.
\end{proof}

\begin{remark}
The inequality in Theorem \ref{theorem:HKS} was already proved by Hirzallah et al. in \cite{HKS} using different technique. If we take $B=I$ in Theorem \ref{theorem:HKS},  then  we have the well known inequality $w(A^*T \pm TA) \leq  2 w(T)\|A\|$, i.e., $w(AT \pm TA^*) \leq  2 w(T)\|A\|$.
\end{remark}

Our final result in this section is to  compute  upper bound for the numerical radius of  a bounded linear operator $T$  in terms of $\|\textit{Re}(T)\|,\|\textit{Im}(T)\|,m(\textit{Re}(T))$ and $m(\textit{Im}(T))$ using Cartesian decomposition.

\begin{theorem}\label{theorem:cartesian}
Let $T\in B(\mathbb{H}).$ Then
\begin{eqnarray*}
w^4(T)&\leq& \max \left \{\big|\|\textit{Re}(T)\|^2-m^2(\textit{Im}(T))\big|^2, \big|\|\textit{Im}(T)\|^2-m^2(\textit{Re}(T))\big|^2 \right\}\\ && + 4\|\textit{Re}(T)\|^2\|\textit{Im}(T)\|^2.
\end{eqnarray*}

\begin{proof}
Let $x \in \mathbb{H}$ with $\|x\|=1.$ Then from the Cartesian decomposition of $T$, we have
\begin{eqnarray*}
\langle Tx,x \rangle &=& \langle \textit{Re}(T)x,x \rangle+i\langle \textit{Im}(T)x,x \rangle\\
\Rightarrow \langle Tx,x \rangle^2 &=& \langle \textit{Re}(T)x,x \rangle^2-\langle \textit{Im}(T)x,x \rangle^2+2i\langle \textit{Re}(T)x,x \rangle \langle \textit{Im}(T)x,x \rangle\\
\Rightarrow \big |\langle Tx,x \rangle^2\big |^2 &=& \big |\langle \textit{Re}(T)x,x \rangle^2-\langle \textit{Im}(T)x,x \rangle^2 \big |^2+4\langle \textit{Re}(T)x,x \rangle ^2 \langle \textit{Im}(T)x,x \rangle ^2\\
 \Rightarrow \big |\langle Tx,x \rangle\big |^4 &\leq& \max \left \{\big|\|\textit{Re}(T)\|^2-m^2(\textit{Im}(T))\big|^2, \big|\|\textit{Im}(T)\|^2-m^2(\textit{Re}(T))\big|^2 \right\} \\ && + 4\|\textit{Re}(T)\|^2\|\textit{Im}(T)\|^2.
\end{eqnarray*}
Taking supremum over $x, \|x\|=1$, we get the desired inequality.
\end{proof}
\end{theorem}

\section{\textbf{Upper bounds for the numerical radius of operator matrices}}

\noindent In this section  we obtain bounds for  the numerical radius of $n\times n$ operator matrices. We begin with  the estimation of an upper bound for the  $n\times n$ operator matrix for which entires of all rows are zero operators except first row. For this we need the following inequality \cite[Remark $2.8$]{BBP}, which gives an  upper bound for the numerical radius of a bounded linear operator $T$ in terms of $\|\textit{Re}(T)\|$ and $\|\textit{Im}(T)\|,$ where $ \textit{Re}(T) = \frac{1}{2} (T + T^*) $  and $ \textit{Im}(T) = \frac{1}{2i} (T - T^*) .$ 

\begin{lemma}\label{lemma:2}
Let $T \in B(\mathbb{H})$. Then 
\begin{eqnarray*}
w^2(T) &\leq& \|\textit{Re}(T)\|^2+\|\textit{Im}(T)\|^2.
\end{eqnarray*}
\end{lemma}
  
\begin{theorem}\label{theorem:upper bound oprt 1}
Let $A_{11}\in B(\mathbb{H}_1,\mathbb{H}_1), A_{12}\in B(\mathbb{H}_2,\mathbb{H}_1),\ldots, A_{1n}\in B(\mathbb{H}_n,\mathbb{H}_1). $ Then
\begin{eqnarray*}
w\left(\begin{array}{cccc}
    A_{11}&A_{12}&\ldots &A_{1n} \\
    0&0&\ldots &0\\
    \vdots & \vdots & &\vdots  \\
     0&0&\ldots&0
    \end{array}\right)&\leq& \frac{1}{2} \sqrt{\alpha^2+\beta^2},
\end{eqnarray*}
where 
\begin{eqnarray*}
\alpha&=&\|\textit{Re}(A_{11})\|+\sqrt{\|\textit{Re}(A_{11})\|^2+\sum_{j=2}^n\|A_{1j}\|^2},\\ 
\beta&=&\|\textit{Im}(A_{11})\|+\sqrt{\|\textit{Im}(A_{11})\|^2+\sum_{j=2}^n\|A_{1j}\|^2}.
\end{eqnarray*}
\end{theorem}

\begin{proof}
Let $T=\left(\begin{array}{cccc}
   A_{11}&A_{12}&\ldots &A_{1n} \\
    0&0&\ldots &0\\
    \vdots & \vdots & &\vdots  \\
     0&0&\ldots&0
    \end{array}\right).$ Then\\
		$\textit{Re}(T)=\left(\begin{array}{cccc}
    \textit{Re}(A_{11})&\frac{A_{12}}{2}&\ldots&\frac{A_{1n}}{2} \\
    \frac{A^{*}_{12}}{2}&0&\dots &0\\
	  \vdots& \vdots& &\vdots \\
    \frac{A^{*}_{1n}}{2}&0&\ldots&0
    \end{array}\right)$ and   
		$\textit{Im}(T)=\left(\begin{array}{cccc}
    \textit{Im}(A_{11})&\frac{A_{12}}{2i}&\ldots&\frac{A_{1n}}{2i} \\
    -\frac{A^{*}_{12}}{2i}&0&\ldots&0\\
	   \vdots& \vdots& &\vdots \\
    -\frac{A^{*}_{1n}}{2i}&0&\ldots&0
    \end{array}\right).$\\
Now, 
\begin{eqnarray*}
	\|\textit{Re}(T)\|&\leq& \left \| \left(\begin{array}{cccc}
    \|\textit{Re}(A_{11})\|&\frac{\|A_{12}\|}{2}&\ldots&\frac{\|A_{1n}\|}{2} \\
    \frac{\|A^{*}_{12}\|}{2}&0&\ldots&0\\
	  
    \vdots& \vdots& &\vdots \\
    \frac{\|A^{*}_{1n}\|}{2}&0&\ldots&0
    \end{array}\right) \right \|\\
		&=& \left \|\left(\begin{array}{cccccc}
    \|\textit{Re}(A_{11})\|&\frac{\|A_{12}\|}{2}&\ldots&\frac{\|A_{1n}\|}{2} \\
    \frac{\|A_{12}\|}{2}&0&\ldots&0\\
	  
    \vdots& \vdots & &\vdots \\
    \frac{\|A_{1n}\|}{2}&0&\ldots&0
    \end{array}\right) \right \|\\
		&=& \frac{1}{2}\left(\|\textit{Re}(A_{11})\|+\sqrt{\|\textit{Re}(A_{11})\|^2+\sum_{j=2}^n\|A_{1j}\|^2}\right).
\end{eqnarray*}	
Similarly, 
\begin{eqnarray*}	
\|\textit{Im}(T)\|&\leq& \frac{1}{2}\left(\|\textit{Im}(A_{11})\|+\sqrt{\|\textit{Im}(A_{11})\|^2+\sum_{j=2}^n\|A_{1j}\|^2}\right). 
\end{eqnarray*}
Using these bounds of $\|\textit{Re}(T)\|$ and $\|\textit{Im}(T)\|$ in Lemma \ref{lemma:2}, we get the desired inequality and this completes the proof of the theorem.
\end{proof}

Next using  Theorem \ref{theorem:upper bound oprt 1}, we compute an upper bound for the numerical radius of arbitrary $n\times n$ operator matrices.

\begin{theorem} \label{theorem:upper bound oprt 2}
 Let $T=(A_{ij})$ be an $n\times n$ operator matrix with $A_{ij}\in B(\mathbb{H}_j,\mathbb{H}_i)$. Then 
\begin{eqnarray*}
w(T)&\leq & \sum^n_{k=1}\frac{1}{2}\sqrt{\alpha^2_k+\beta^2_k},
\end{eqnarray*}
where
\begin{eqnarray*}
\alpha_k &=&\|\textit{Re}(A_{kk})\|+\sqrt{\|\textit{Re}(A_{kk})\|^2+\sum_{j=1,k\neq j}^n\|A_{kj}\|^2},\\ 
\beta_k &=&\|\textit{Im}(A_{kk})\|+\sqrt{\|\textit{Im}(A_{kk})\|^2+\sum_{j=1,k\neq j}^n\|A_{kj}\|^2}.
\end{eqnarray*}
\end{theorem}

\begin{proof}
Applying triangle inequality for numerical radius, we have\\
\[w(T)\leq w(T_1)+w(T_2)+\ldots+w(T_n),\] where 
\begin{eqnarray*}
&& T_1=\left(\begin{array}{cccc}
    A_{11}&A_{12}&\ldots&A_{1n} \\
    0&0&\ldots&0\\
    \vdots& \vdots & &\vdots \\
    0&0&\ldots&0
    \end{array}\right),
		T_2=\left(\begin{array}{cccc}
    0&0&\ldots &0\\
		A_{21}&A_{22}&\ldots &A_{2n} \\
    \vdots& \vdots& &\vdots \\
    0&0& &0
    \end{array}\right),\\
		&& \ldots, T_n=\left(\begin{array}{cccccc}
    0&0&\ldots &0\\
		0&0&\ldots &0 \\
    \vdots & \vdots & &\vdots \\
		A_{n1}&A_{n2}&\ldots &A_{nn}
    \end{array}\right).
	\end{eqnarray*}	
\noindent For each $i=2,3,\ldots,n$, let $U_i$ be the unitary operator matrix obtained by interchanging $1$st and $i$th column of $n\times n$ identity operator matrix. Therefore using weak unitary invariance property of the numerical radius, i.e., $w(U^*TU)=w(T)$ for any unitary operator $U$, we have
\[w(T)\leq w(T_1)+w(U^*_2T_2U_2)+w(U^*_3T_3U_3)+\ldots+w(U^*_nT_nU_n).\] This gives 

\begin{eqnarray*}
w(T) &\leq & w\left(\begin{array}{cccc}
    A_{11}&A_{12}&\ldots &A_{1n} \\
    0&0&\ldots &0\\
    \vdots & \vdots & &\vdots \\
    0&0&\ldots &0
    \end{array}\right)+w\left(\begin{array}{cccccc}
    A_{22}&A_{21}&\ldots &A_{2n} \\
    0&0&\ldots &0\\
    \vdots & \vdots & &\vdots \\
    0&0&\ldots &0
    \end{array}\right)\\
		&&+\ldots+w\left(\begin{array}{cccccc}
    A_{nn}&A_{n2}&\ldots &A_{n1} \\
    0&0&\ldots &0\\
    \vdots & \vdots & &\vdots \\
    0&0&\ldots &0
    \end{array}\right).
\end{eqnarray*}
\noindent Therefore applying Theorem \ref{theorem:upper bound oprt 1}, we get the desired inequality and this completes the proof of the theorem.
\end{proof}

Next we obtain new upper bounds for the numerical radius of $2\times 2$ operator matrices. For this we need the following equality \cite{Y} by Yamazaki.

\begin{lemma} \label{theorem:lemma3}
	Let $T \in B(H)$, then 
	\begin{eqnarray*}
	w(T)&=&\sup_{\theta \in \mathbb{R}}\| \textit {Re}(e^{i \theta}T) \|. 
\end{eqnarray*}
By replacing $T$ by $iT$ in the above equality, also we have
\begin{eqnarray*}
	w(T)&=&\sup_{\theta \in \mathbb{R}}\| \textit {Im}(e^{i \theta}T) \|. 
\end{eqnarray*}
\end {lemma}	

\begin{theorem}\label{theorem:upper 2}
Let $T=\left(\begin{array}{cc}
A&B\\
0&0
\end{array}\right),$ where $A\in B(\mathbb{H}_1), B\in B(\mathbb{H}_2,\mathbb{H}_1).$ Then
\[w(T)\leq \sqrt{w^2(A)+\frac{1}{2}\|B\|\left(w(A)+\frac{1}{2}\|B\|\right)}.\]
\end{theorem}
\begin{proof}
From an easy calculation we have, for every $\theta \in \mathbb{R}$ 
\begin{eqnarray*}
\textit{Re}(e^{i\theta}T)&=&\left(\begin{array}{cc}
\textit{Re}(e^{i\theta}A)&\frac{1}{2}e^{i\theta}B\\
\frac{1}{2}e^{-i\theta}B^{*}&0
\end{array}\right)\\
&=& \left(\begin{array}{cc}
\textit{Re}(e^{i\theta}A)&0\\
0&0
\end{array}\right)+\left(\begin{array}{cc}
0&\frac{1}{2}e^{i\theta}B\\
\frac{1}{2}e^{-i\theta}B^{*}&0
\end{array}\right).
\end{eqnarray*}
This implies that 
\begin{eqnarray*}
\big(\textit{Re}(e^{i\theta}T)\big)^2&=& \left(\begin{array}{cc}
(\textit{Re}(e^{i\theta}A))^2&0\\
0&0
\end{array}\right)+\left(\begin{array}{cc}
\frac{1}{4}BB^*&0\\
0&\frac{1}{4}B^*B
\end{array}\right)\\
&&+ \left(\begin{array}{cc}
0&\frac{1}{2}e^{i\theta} \textit{Re} (e^{i\theta}A)B\\
0&0
\end{array}\right)+\left(\begin{array}{cc}
0&0\\
\frac{1}{2}e^{-i\theta}B^* \textit{Re} (e^{i\theta}A)&0
\end{array}\right).
\end{eqnarray*}
Therefore, 
\begin{eqnarray*}
\|\textit{Re}(e^{i\theta}T)\|^2 &\leq& \|\textit{Re}(e^{i\theta}A)\|^2+\frac{1}{4}\|B\|^2+\frac{1}{2}\|\textit{Re}(e^{i\theta}A)\| \|B\|\\
            &\leq& w^2(A)+\frac{1}{4}\|B\|^2+\frac{1}{2}w(A) \|B\|.
\end{eqnarray*}
Taking supremum over $\theta$, we get 
\[w^2(T)\leq w^2(A)+\frac{1}{4}\|B\|^2+\frac{1}{2}w(A) \|B\|.\]
This completes the proof of the theorem.
\end{proof}

Now using Theorem \ref{theorem:upper 2}, we give an upper bound for the numerical radius of arbitrary $2\times 2$ operator matrices.

\begin{cor}\label{corollary:upper A3}
Let $T=\left(\begin{array}{cc}
A&B\\
C&D
\end{array}\right),$ where $A\in B(\mathbb{H}_1), B\in B(\mathbb{H}_2,\mathbb{H}_1), C\in B(\mathbb{H}_1,\mathbb{H}_2),D\in B(\mathbb{H}_2).$ Then
\[w(T)\leq \sqrt{w^2(A)+\frac{1}{2}\|B\|\left(w(A)+\frac{1}{2}\|B\|\right)}+\sqrt{w^2(D)+\frac{1}{2}\|C\|\left(w(D)+\frac{1}{2}\|C\|\right)}.\]
\end{cor}
\begin{proof}
We consider an unitary operator matrix $U=\left(\begin{array}{cc}
0&I\\
I&0
\end{array}\right)$ and using weak unitary invariance property of the numerical radius, we have 

\begin{eqnarray*}
w(T)&\leq& w\left(\begin{array}{cc}
A&B\\
0&0
\end{array}\right)+w\left(\begin{array}{cc}
0&0\\
C&D
\end{array}\right)\\
&=& w\left(\begin{array}{cc}
A&B\\
0&0
\end{array}\right)+w\left(U^*\left(\begin{array}{cc}
0&0\\
C&D
\end{array}\right)U\right)\\
&=& w\left(\begin{array}{cc}
A&B\\
0&0
\end{array}\right)+w\left(\begin{array}{cc}
D&C\\
0&0
\end{array}\right).
\end{eqnarray*}
Therefore, applying Theorem \ref{theorem:upper 2} we get, the required inequality of the theorem.
\end{proof}

In the following theorem we provide a new upper bound for $ 2 \times 2$ operator matrices, in which the entries in second row are all zero operators. 

\begin{theorem}\label{theorem:upper 4}
Let $T=\left(\begin{array}{cc}
A&B\\
0&0
\end{array}\right),$ where $A\in B(\mathbb{H}_1), B\in B(\mathbb{H}_2,\mathbb{H}_1).$ Then
\[w(T)\leq \sqrt{2w^2(A)+\frac{1}{2}\left(\|A^*B\|+\|B\|^2\right)}.\]
\end{theorem}
\begin{proof}
For  $\theta \in \mathbb{R},$ it is easy to see that \\ \\
$\textit{Re}(e^{i\theta}T)=\left(\begin{array}{cc}
\textit{Re}(e^{i\theta}A)&\frac{1}{2}e^{i\theta}B\\
\frac{1}{2}e^{-i\theta}B^{*}&0
\end{array}\right)$ and 
$ \textit{Im}(e^{i\theta}T)=-i\left(\begin{array}{cc}
i\textit{Im}(e^{i\theta}A)&\frac{1}{2}e^{i\theta}B\\
-\frac{1}{2}e^{-i\theta}B^{*}&0
\end{array}\right).$\\
Therefore, from simple calculation, we have \\
\begin{eqnarray*}
\textit{Re}^2(e^{i\theta}T)+\textit{Im}^2(e^{i\theta}T)&=&\left(\begin{array}{cc}
\textit{Re}^2(e^{i\theta}T)+\textit{Im}^2(e^{i\theta}T)&0\\
0&0
\end{array}\right)+\left(\begin{array}{cc}
0&\frac{A^*B}{2}\\
\frac{B^*A}{2}&0
\end{array}\right)\\
&+&\left(\begin{array}{cc}
\frac{BB^*}{2}&0\\
0&\frac{B^*B}{2}
\end{array}\right).
\end{eqnarray*}
Since $\textit{Im}^2(e^{i\theta}T)\geq 0$, so we get,
\begin{eqnarray*}
\textit{Re}^2(e^{i\theta}T)&\leq&\left(\begin{array}{cc}
\textit{Re}^2(e^{i\theta}T)+\textit{Im}^2(e^{i\theta}T)&0\\
0&0
\end{array}\right)+\left(\begin{array}{cc}
0&\frac{A^*B}{2}\\
\frac{B^*A}{2}&0
\end{array}\right)\\
&+&\left(\begin{array}{cc}
\frac{BB^*}{2}&0\\
0&\frac{B^*B}{2}
\end{array}\right).
\end{eqnarray*}
Taking norm on both sides we get,
\begin{eqnarray*}
\|\textit{Re}(e^{i\theta}T)\|^2 &\leq&  \|\textit{Re}^2(e^{i\theta}A)+\textit{Im}^2(e^{i\theta}A)\|+\frac{1}{2}\|A^*B\|+\frac{1}{2}\|B\|^2.\\
   &\leq& 2w^2(A)+\frac{1}{2} \left(\|A^*B\|+\|B\|^2 \right).
\end{eqnarray*}
Taking supremum over $\theta \in \mathbb{R}$ we get, 
\[w^2(T)\leq 2w^2(A)+\frac{1}{2} \left(\|A^*B\|+\|B\|^2 \right).\]
This completes the proof.
\end{proof}

Now, using Theorem \ref{theorem:upper 4} and using the same technique as in the proof of Corollary \ref{corollary:upper A3}, we can obtain the following bound for numerical radius of any $2\times 2$ operator matrices.

\begin{cor}\label{corollary:upper 3}
Let $T=\left(\begin{array}{cc}
A&B\\
C&D
\end{array}\right),$ where $A\in B(\mathbb{H}_1), B\in B(\mathbb{H}_2,\mathbb{H}_1), C\in B(\mathbb{H}_1,\mathbb{H}_2),D\in B(\mathbb{H}_2).$ Then
\[w(T)\leq \sqrt{2w^2(A)+\frac{1}{2} \left( \|A^*B\|+\|B\|^2  \right)}+\sqrt{2w^2(D)+\frac{1}{2} \left(\|D^*C\|+\|C\|^2 \right)}.\]
\end{cor}

\begin{remark}
	Considering the operator $T=\left(\begin{array}{cc}
	A&B\\
	0&0
	\end{array}\right),$ where $A=\left(\begin{array}{cc}
	0&0\\
	3&1
	\end{array}\right) $ and $  B=\left(\begin{array}{cc}
	1&2\\
	0&0
	\end{array}\right)$, it is easy to see that Theorem \ref{theorem:upper 4} gives $w(T)\leq \sqrt{8+\sqrt{10}}$ whereas  the bound obtained by Shebrawi in \cite [Th. 3.2]{S} gives $w(T)\leq \frac{1}{4}(12+\sqrt{10}).$ This indicates that for this operator the bound obtained by us is better than that obtained by Shebrawi.
\end{remark}

\section{\textbf{Lower bounds for the numerical radius of operator matrices}}

\noindent In  this section  we  first obtain a new lower bound for the numerical radius of a special class of $n\times n$ operator matrices.

\begin{theorem}\label{theorem:lower 1}
Let $T=\left(\begin{array}{ccccc}
0&0&\ldots&0&A_1\\
0&0&\ldots&A_2&0\\
\vdots&\vdots& &\vdots&\vdots\\
A_n&0&\ldots&0&0\\
\end{array}\right),$ where $A_i \in B(\mathbb{H})$ for each $i=1,2,\ldots,n$. Then
\[w(T)\geq \frac{1}{\sqrt{2}}\max_{1\leq i \leq n} \left\{\sqrt{w(A_iA_{n-i+1}+A_{n-i+1}A_i)},\sqrt{w(A_iA_{n-i+1}-A_{n-i+1}A_i)}\right\}.\]
\end{theorem}

\begin{proof}
Consider the unitary operator $U=\left(\begin{array}{ccccc}
0&0&\ldots&0&I\\
0&0&\ldots&I&0\\
\vdots&\vdots& &\vdots&\vdots\\
I&0&\ldots&0&0\\
\end{array}\right)$. \\
Then it is easy to see that,
\begin{eqnarray*}
&&T^2+(U^*TU)^2= \\
&&\left(\begin{array}{ccccc}
A_1A_n+A_nA_1&0&\ldots&0&0\\
0&A_2A_{n-1}+A_{n-1}A_2&\ldots&0&0\\
\vdots&\vdots& &\vdots&\vdots\\
0&0&\ldots&0&A_nA_1+A_1A_n\\
\end{array}\right)=D_1~~ \mbox{(say)}.
\end{eqnarray*}
Therefore, 
\begin{eqnarray*}
w(D_1)&=& w(T^2+(U^*TU)^2)\\
&\leq& w(T^2)+w \big((U^*TU)^2)\big)\\
&\leq & w^2(T)+w^2(U^*TU)\\
&=& 2w^2(T),~~\mbox{ by weak unitary invariance. }
\end{eqnarray*}
This shows that 
\[ \max \left\{w(A_iA_{n-i+1}+A_{n-i+1}A_i):1\leq i \leq n\right\} \leq 2w^2(T).\]
Now, we calculate $T^2-(U^*TU)^2$ and then using the same argument as above we can prove that 
\[\max \left\{w(A_iA_{n-i+1}-A_{n-i+1}A_i):1\leq i \leq n\right\} \leq 2w^2(T).\]
Therefore we conclude that 
\[w(T)\geq \frac{1}{\sqrt{2}}\max_{1\leq i \leq n} \left\{\sqrt{w(A_iA_{n-i+1}+A_{n-i+1}A_i)},\sqrt{w(A_iA_{n-i+1}-A_{n-i+1}A_i)}\right\}.\]
\end{proof}

Now using  Theorem \ref{theorem:lower 1} and the pinching inequalities (see \cite[p. 107]{B}),
 \[w\left(\begin{array}{cc}
A&B\\
C&D
\end{array}\right)\geq w\left(\begin{array}{cc}
A&0\\
0&D
\end{array}\right) ~~\mbox{and}~~ w\left(\begin{array}{cc}
A&B\\
C&D
\end{array}\right)\geq w\left(\begin{array}{cc}
0&B\\
C&0
\end{array}\right),\] where $A, B, C, D\in B(\mathbb{H})$, we obtain the following lower bound for the numerical radius of arbitrary $2\times2$ operator matrices.

\begin{cor}\label{corollary:lower 2}
Let $T=\left(\begin{array}{cc}
A&B\\
C&D
\end{array}\right),$ where $A, B, C, D\in B(\mathbb{H})$. Then
\[w(T)\geq \max \left \{w(A), w(D), \sqrt{\frac{1}{2}w(BC+CB)},\sqrt{\frac{1}{2}w(BC-CB)}\right \}.\]
\end{cor}
\begin{remark}
The inequality obtained in Corollary \ref{corollary:lower 2} and the first inequality in \cite [Th. 3.7]{HKS} obtained by Hirzallah et al. are incomparable. Consider  $T=\left(\begin{array}{cc}
A&B\\
C&D
\end{array}\right),$ where $A=D=(0), B=(1),C=(2).$ Then Corollary \ref{corollary:lower 2} gives $w(T)\geq \sqrt{2}$ and \cite [Th. 3.7]{HKS} gives $w(T)\geq \frac{3}{2}.$ Again, if  we consider $T=\left(\begin{array}{cc}
A&B\\
C&D
\end{array}\right),$ where $A=D=\left(\begin{array}{cc}
0&0\\
0&0
\end{array}\right), B=\left(\begin{array}{cc}
-1&3\\
0&1
\end{array}\right),C=\left(\begin{array}{cc}
1&3\\
0&-1
\end{array}\right),$ then Corollary \ref{corollary:lower 2} gives $w(T)\geq \sqrt{3}$ and \cite [Th. 3.7]{HKS} gives $w(T)\geq \frac{3}{2}.$
\end{remark}

We next prove an inequality which gives a lower  bound for the  numerical radius of  $2 \times 2$ operator matrices of the form $\left(\begin{array}{cc}
A&B\\
0&0
\end{array}\right)$, where $A, B\in B(\mathbb{H}).$ To do so  we need the following lemma which follows from weak unitary invariance property of the numerical radius.

\begin{lemma}\label{lemma:3}
	Let $T=\left(\begin{array}{cc}
	A&B\\
	B&A
	\end{array}\right),$ where $A, B\in B(\mathbb{H})$. Then
	\begin{eqnarray*}
		w(T) &= & \max \left \{w(A+B), w(A-B) \right \}.
	\end{eqnarray*}
	\end {lemma}
	
	Now we prove the theorem.
	\begin{theorem}
		Let 	$A, B\in B(\mathbb{H}).$ Then 
		\[ w\left(\begin{array}{cc}
		A&B\\
		0&0
		\end{array}\right) \geq  \frac{1}{2} \max \left \{w(A+B), w(A-B) \right \}.\]
		
	\end{theorem}
	\begin{proof}
		Let $T=\left(\begin{array}{cc}
		A&B\\
		0&0
		\end{array}\right).$ We consider an unitary operator matrix, $U=\left(\begin{array}{cc}
		0&I\\
		I&0
		\end{array}\right).$ Then  we get,
		\begin{eqnarray*}
			\left(\begin{array}{cc}
				A&B\\
				B&A
			\end{array}\right)&=& T+U^*TU\\
			\Rightarrow w\left(\begin{array}{cc}
				A&B\\
				B&A
			\end{array}\right) &\leq & w(T)+w(U^*TU)\\
			&=& 2w(T), ~~\mbox {by weak unitary invariance}\\
			\Rightarrow \max \left \{w(A+B), w(A-B) \right \}&\leq& 2w(T), ~~\mbox{using Lemma \ref{lemma:3}}.
		\end{eqnarray*}
		This completes the proof.
	\end{proof}

We end this section with the following theorem, in which we obtain an inequality for the lower bound of numerical radius of $2\times 2$ operator matrix,  which generalizes the inequality  $ w(T) \geq \| Re(T) \| $ and $ w(T) \geq \|Im(T) \|,$ obtained by  Kittaneh et al. \cite{KMY}.

\begin{theorem}\label{theorem:lower 3}
Let $T=\left(\begin{array}{cc}
0&A\\
B&0
\end{array}\right),$ where $A, B\in B(\mathbb{H})$. Then
\[w(T)\geq \frac{1}{2} \sup_{\theta \in \mathbb{R}}\left \| \textit{Re}(e^{i\theta}A) \pm \textit{Re}(e^{i\theta}B)\right \|,\]
 
\[w(T)\geq \frac{1}{2} \sup_{\theta \in \mathbb{R}}\left \| \textit{Im}(e^{i\theta}A) \pm \textit{Im}(e^{i\theta}B)\right \|.\]
\end{theorem}

\begin{proof}
Let $H_{\theta}=\textit{Re}(e^{i\theta}T)$ and  $U=\left(\begin{array}{cc}
0&I\\
I&0
\end{array}\right)$ be an unitary operator. Then  we get,
\[H_{\theta}+U^*H_{\theta}U=\left(\begin{array}{cc}
0&\textit{Re}(e^{i\theta}A)+\textit{Re}(e^{i\theta}B)\\
\textit{Re}(e^{i\theta}A)+\textit{Re}(e^{i\theta}B)&0
\end{array}\right).\]
Taking norm on both sides we get,
\begin{eqnarray*}
\|\textit{Re}(e^{i\theta}A)+\textit{Re}(e^{i\theta}B)\|&=&\|H_{\theta}+U^*H_{\theta}U\|\\
&\leq &\|H_{\theta}\|+\|U^*H_{\theta}U\|\\
&= &2\|H_{\theta}\|\\
&\leq& 2w(T).
\end{eqnarray*}
Since this holds for all $\theta \in \mathbb{R}$, so we have
\[w(T)\geq \frac{1}{2} \sup_{\theta \in \mathbb{R}}\left \| \textit{Re}(e^{i\theta}A)+\textit{Re}(e^{i\theta}B)\right \|.\]
 Next we consider $K_{\theta}=\textit{Im}(e^{i\theta}T).$ Then we get, 
\[K_{\theta}+U^*K_{\theta}U=\left(\begin{array}{cc}
0&\textit{Im}(e^{i\theta}A)+\textit{Im}(e^{i\theta}B)\\
\textit{Im}(e^{i\theta}A)+\textit{Im}(e^{i\theta}B)&0
\end{array}\right).\]
Taking norm on both sides we get,
\begin{eqnarray*}
\|\textit{Im}(e^{i\theta}A)+\textit{Im}(e^{i\theta}B)\|&=&\|K_{\theta}+U^*K_{\theta}U\|\\
&\leq &\|K_{\theta}\|+\|U^*K_{\theta}U\|\\
&= &2\|K_{\theta}\|\\
&\leq& 2w(T).
\end{eqnarray*}
Since this holds for all $\theta \in \mathbb{R}$, so we have
\[w(T)\geq \frac{1}{2} \sup_{\theta \in \mathbb{R}}\left \| \textit{Im}(e^{i\theta}A)+\textit{Im}(e^{i\theta}B)\right \|.\]
\noindent Considering $H_{\theta}-U^*H_{\theta}U$ and $K_{\theta}-U^*K_{\theta}U$   and using similar arguments as above we can prove the remaining inequalities.
\end{proof}

\begin{remark}
If we take $A=B$ and $\theta=0$ in  Theorem \ref{theorem:lower 3}, then we get,  $w(A)\geq \|\textit{Re}(A)\|$ and $w(A)\geq \|\textit{Im}(A)\|.$ 
\end{remark}

\begin{remark}
There was a minor error in the calculation of bound in Remark 2.4 of \cite{BBP}, the estimation of bound obtained there should be $ 1.86317171 $ instead of $1.784$.  This was pointed out by the reviewer while reviewing the paper for Mathematical  Reviews (MR3933295), we thank him/her for that.
\end{remark}

\bibliographystyle{amsplain}

\end{document}